\documentclass[a4paper,reqno]{amsart}
\usepackage{amssymb, amsmath, amscd}
\usepackage{color}

\newtheorem{theorem}{Theorem}[section]

\newtheorem{lemma}{Lemma}[section]
\newtheorem{remark}{Remark}[section]

\makeatletter
 \@addtoreset{equation}{section}
\makeatother
\begin{document}
\title{Universal curvature identities II}
\author{P. Gilkey, J.H. Park, and K. Sekigawa}
\address{PG: Mathematics Department, University of Oregon, Eugene OR 97403 USA}
\email{gilkey@uoregon.edu}
\address{JHP: Department of Mathematics, Sungkyunkwan University, Suwon
440-746, Korea \\\phantom{JHP:..A} Korea Institute for Advanced Study, Seoul
130-722, Korea} \email{parkj@skku.edu}
\address{KS: Department of Mathematics, Niigata University, Niigata, Japan.}
\email{sekigawa@math.sc.niigata-u.ac.jp}
\begin{abstract}{We
show that any universal curvature identity which holds in the Riemannian
setting extends naturally to the
pseudo-Riemannian setting. Thus the
Euh-Park-Sekigawa identity also holds for pseudo-Riemannian manifolds. We study the Euler-Lagrange
equations associated to the Chern-Gauss-Bonnet formula and show that as in the Riemannian setting,
they are given solely in terms of curvature (and not in terms of covariant derivatives of curvature) even
in the pseudo-Riemannian setting.
\\MSC 2010: 53B20.
\\Keywords: Pfaffian, Chern-Gauss-Bonnet theorem, Euler-Lagrange Equations,
Euh-Park-Sekigawa identity.}\end{abstract}

\maketitle

\section{Introduction}

The study of scalar and symmetric 2-tensor valued invariants of a
metric is central in modern differential geometry. It also plays an
important role in mathematical physics. The scalar curvature is the
simplest such invariant and plays a central role not only in the
Riemannian geometry \cite{D10,M10}. It also is important in the
higher signature setting \cite{CL11,K11,M11,ZW10}.  The norm of the
Weyl conformal curvature tensor $|W|^2$ appears in many settings,
see for example the discussion in \cite{C11,H10,K10}. Turning to
symmetric 2-tensor valued invariants, the trace free Ricci tensor is
important \cite{G10} as is the Ricci tensor not only in the positive
definite \cite{M10x} but also the indefinite settings
\cite{CRR10,Ca10}; see also \cite{SKY06} where the Weyl conformal
tensor plays a crucial role. The Pfaffian (Gauss-Bonnet curvature)
is a complicated invariant of the curvature tensor that defines the
Einstein-Hilbert-Lovelock functional \cite{LS10} and is important in
Kazdan-Warner type identities \cite{GHL11}. It is also related to
the Lipschitz-Killing curvature \cite{SX11}.

Motivated by these example (and many more), we have decided to undertake a
systematic study of scalar and symmetric 2-tensor valued invariants from an abstract
point of view not only in the Riemannian but also the higher signature setting. We
first proceed in the purely algebraic setting. A pair
$(V,\varepsilon)$ is called an {\it inner product space} if $V$ is real vector space
of dimension $m$ and if
$\varepsilon$ is a non-degenerate inner product of signature $(p,q)$ on $V$ where $p+q=m$. An {\it algebraic curvature
tensor}
$A$ is an element of
$\otimes^4(V^*)$ satisfying the relations of the Riemann curvature tensor, namely for all
$x,y,z,w\in V$ one has the following relations:
\begin{equation}\label{eqn-1.a}
\begin{array}{l}
A(x,y,z,w)=-A(y,x,z,w)=A(z,w,x,y),\\
A(x,y,z,w)+A(y,z,x,w)+A(z,x,y,w)=0\,.\vphantom{\vrule height 12pt}
\end{array}\end{equation}
Let $\mathfrak{A}(V)\subset\otimes^4(V^*)$
be the linear subspace of all such tensors. A triple $(V,\varepsilon,A)$ is said to be a {\it curvature
model} if $(V,\varepsilon)$ is an inner product space and if $A\in\mathfrak{A}(V)$.

\subsection{Geometric Realizations}
We say that a curvature model $(V,\varepsilon,A)$ is {\it geometrically realized} at a
point
$\xi$ of a pseudo-Riemannian manifold $(M,g)$ if there exists an
isomorphism
$\Phi$ from $T_\xi M$ to $V$ so that $\Phi^*\varepsilon=g_\xi$ is  the metric at $\xi$
and so that
$\Phi^*A=R_\xi$ is the associated Riemann curvature tensor of the Levi-Civita connection at $\xi$.
A useful result in the field, which we shall prove in Section
\ref{sect-2} in the interests of completeness, is the following result:

\begin{theorem}\label{thm-1.1} Every curvature model can be geometrically realized at some point of some
compact pseudo-Riemannian manifold.
\end{theorem}

Theorem~\ref{thm-1.1} shows that the symmetries of Equation~(\ref{eqn-1.a}) generate the
universal curvature symmetries of the Riemann curvature tensor; there are no
hidden additional symmetries. It lets us pass freely between the algebraic context and the geometric
setting.

\subsection{Scalar invariants} The orthogonal group
$\mathcal{O}(V,\varepsilon)$ acts on the space of algebraic curvature tensors
$\mathfrak{A}(V)$ by pullback where one sets:
$$(T^*A)(v_1,v_2,v_3,v_4):=A(Tv_1,Tv_2,Tv_3,Tv_4)\,.$$
If $P(A)$ is a polynomial of degree $\ell$ in the components of
$A$ relative to some basis for $V$, then we shall say that $P$ is a {\it scalar invariant} if
$P(T^*A)=P(A)$ for all
$A\in\mathfrak{A}$ and for all $T\in\mathcal{O}(V,\varepsilon)$. Let
$\mathcal{O}(V,\varepsilon)$ act trivially on $\mathbb{R}$. If we polarize
$P$, we can regard, equivalently,
$P$ as a {\it linear invariant} of $\otimes^\ell(\mathfrak{A}(V))$, i.e. as an equivariant linear map from
$\otimes^\ell(\mathfrak{A}(V))$ to
$\mathbb{R}$. We
let
$\mathcal{J}_{2\ell}(V,\varepsilon)$ be the vector space of all such maps; these maps are homogeneous of degree $2\ell$
in the derivatives of the metric and this is a convenient indexing convention to use as there will be maps of odd order in
the derivatives of the metric when we study manifolds with boundary presently.

Since any two inner product spaces of the same
signature
$(p,q)$ are isomorphic, we shall set
$\mathcal{J}_{2\ell}(p,q):=\mathcal{J}_{2\ell}(V,\varepsilon)$ for any inner product space $(V,\varepsilon)$
of signature
$(p,q)$. We shall see presently in Remark~\ref{rmk-1.1} that there is a natural way to identify $\mathcal{J}_{2\ell}(p,q)$
with $\mathcal{J}_{2\ell}(p_1,q_1)$ if $p_1+q_1=p+q$; this common space of invariants will then be denoted by
$\mathcal{J}_{2\ell}(m)$ since only the underlying dimension $m=p+q$ is normative.

Suppose that $P\in\mathcal{J}_{2\ell}(p,q)$. Let $(V,\varepsilon)$ have signature $(p,q)$, let
$(M,g)$ be a pseudo-Riemannian manifold of signature $(p,q)$, and let $R$ be the Riemann curvature
tensor of the Levi-Civita connection on $(M,g)$. Given $\xi$ in $M$, we can find an isometry $\Phi$
which identifies
$(T_\xi M,g_\xi)$ with $(V,\varepsilon)$. We then define $P(R_\xi):=P(\Phi^*R_\xi)$; the particular
isometry $\Phi$ which is chosen being irrelevant as $P$ is invariant under the action of the orthogonal group. Thus
elements of
$\mathcal{J}_{2\ell}$ give rise to geometric invariants; conversely, Theorem \ref{thm-1.1} lets us extend invariants from
the geometric to the algebraic setting and identify the two contexts.

Let $\varepsilon^*$ denote the dual inner product on $V^*$;
$\varepsilon^*$ is a linear map from $V^*\otimes V^*$ to
$\mathbb{R}$ and, more generally, $\otimes^{2\ell}(\varepsilon^*)$is a linear map from $\otimes^{4\ell}(V^*)$ to $\mathbb{R}$ which is
invariant under the action of the orthogonal group. Let
$\operatorname{Perm}(4\ell)$ be the group of permutations of $4\ell$
elements and let $\sigma\in\operatorname{Perm}(4\ell)$. We let
$\sigma$ act on $\otimes^{4\ell}(V^*)$ by permuting the factors and
let
$Q_{2\ell,\sigma,\varepsilon}:=\left\{\otimes^{2\ell}(\varepsilon^*)\right\}\circ\sigma$.
This is clearly invariant under the orthogonal group and the
restriction of $Q_{2\ell,\sigma,\varepsilon}$ to
$\otimes^\ell(\mathfrak{A}(V))$ defines an element of
$\mathcal{J}_{2\ell}(V,\varepsilon)$ which we shall denote by
$P_{2\ell,\sigma,\varepsilon}$.

There is a convenient formalism for describing the invariants $P_{2\ell,\sigma,\varepsilon}$. Choose a basis
$\{e_i\}$ for $V$; let $\{e^i\}$ be the corresponding dual basis for $V^*$. Let
$\varepsilon_{ij}:=\varepsilon(e_i,e_j)$ and let $\varepsilon^{ij}=\varepsilon^*(e^i,e^j)$ denote the
components of $\varepsilon$ and of $\varepsilon^*$ relative to this basis, respectively;
$\varepsilon^{ij}$ is the inverse of the matrix $\varepsilon_{ij}$. If
$x$ and
$y$ are vectors (or co-vectors), we let $x\circ y:=\frac12(x\otimes y+y\otimes x)$ be the {\it
symmetric tensor product} of $x$ with $y$. Adopt the {\it
Einstein convention} and sum over repeated indices to express:
$$\varepsilon=\varepsilon_{ij}e^i\circ e^j\quad\text{and}\quad \varepsilon^*=\varepsilon^{ij}e_i\circ e_j\,.$$If $A\in\otimes^4(V^*)$, let
$A_{ijkl}:=A(e_i,e_j,e_k,e_l)$ give the components of $A$ relative
to the given basis. We then have that:
$$A=A_{ijkl}e^i\otimes e^j\otimes e^k\otimes e^l\,.$$
Let $\sigma\in\operatorname{Perm}(4\ell)$. Set $\sigma_i:=\sigma(i)$ for
$1\le i\le 4\ell$. We may now express:
\begin{equation}\label{eqn-1.b}
P_{2\ell,\sigma,\varepsilon}(A)=\varepsilon^{i_1i_2}\cdot\cdot\cdot\varepsilon^{i_{4\ell-1}i_{4\ell}}
A_{i_{\sigma_1}i_{\sigma_2}i_{\sigma_3}i_{\sigma_4}}\cdot\cdot\cdot
A_{i_{\sigma_{4\ell-3}}i_{\sigma_{4\ell-2}}i_{\sigma_{4\ell-1}}i_{\sigma_{4\ell}}}\,.
\end{equation}
If we let $\beta=\sigma^{-1}$, then we may also express this invariant in the form:
\begin{equation}\label{eqn-1.c}
P_{2\ell,\sigma,\varepsilon}(A)=\varepsilon^{i_{\beta_1}i_{\beta_2}}\cdot\cdot\cdot
\varepsilon^{i_{\beta_{{4\ell-1}}}i_{\beta_{4\ell}}}
A_{i_1i_2i_2i_4}\cdot\cdot\cdot
A_{i_{4\ell-3}i_{4\ell-2}i_{4\ell-1}i_{4\ell}}\,.
\end{equation}

The discussion above shows the value of Equation~(\ref{eqn-1.b}) is independent of the particular basis
which was chosen. The usual scalar invariants of Riemannian geometry can be expressed in this notation.
 The scalar curvature
$\tau$ is given by setting
$$\tau(A):=\varepsilon^{i_1i_2}\varepsilon^{i_3i_4}A_{i_1i_3i_4i_2}\,.$$
Thus $\tau=P_{2\ell,\sigma,\varepsilon}$ where $\sigma_1=1$, $\sigma_2=3$, $\sigma_3=4$, and $\sigma_4=2$,
i.e.
$$
   \sigma=\left(\begin{array}{llll}1234\\1342\end{array}\right)\,.
$$
Note that the permutation defining $\tau$ is not unique as the scalar curvature
can also defined by setting:
$$\tau=\varepsilon^{i_1i_2}\varepsilon^{i_3i_4}A_{i_2i_4i_3i_1}\quad\text{i.e. by taking}\quad
\tilde\sigma:=\left(\begin{array}{llll}1234\\2431\end{array}\right)\,.
$$
The following well known result follows from the Theorem of Invariants of
H. Weyl
\cite{W46}; we will give the proof Section~\ref{sect-3} in the interests of completeness as there is a
minor technical point (see Lemma~\ref{lem-3.1}) which will be needed subsequently:
\begin{theorem}\label{thm-1.2}
$\mathcal{J}_{2\ell}(V,\varepsilon)=\operatorname{Span}_{\sigma\in\operatorname{Perm}(4\ell)}\{P_{2\ell,\sigma,\varepsilon}\}$.
\end{theorem}

\subsection{Universal curvature identities} Let
$\mathcal{C}:=\{c_\sigma\}_{\sigma\in\operatorname{Perm}(4\ell)}$ be a
collection of real constants. We say that
$$P_{2\ell,\mathcal{C},\varepsilon}:=\sum_{\sigma\in\operatorname{Perm}(4\ell)} c_\sigma
P_{2\ell,\sigma,\varepsilon}$$ is a {\it universal curvature identity} on
$(V,\varepsilon)$ if
$P_{2\ell,\mathcal{C},\varepsilon}(A)=0$ for all
$A\in\mathfrak{A}(V)$. Some identities hold for all dimensions. For example, the discussion
given above shows
$$0=\varepsilon^{i_1i_2}\varepsilon^{i_3i_4}A_{i_2,i_4,i_3,i_1}
-\varepsilon^{i_1i_2}\varepsilon^{i_3i_4}A_{i_1,i_3,i_4,i_2}\,.$$
Here
$$c_\sigma=\left\{\begin{array}{rrr}
+1&\text{if}&\sigma=\left(\begin{array}{llll}1234\\2431\end{array}\right)\\
-1&\text{if}&\sigma=\left(\begin{array}{llll}1234\\1342\end{array}\right)^{\vphantom{A}}_{\vphantom{A}}\\
0&\text{if}&\text{otherwise}\end{array}\right\}\,.$$
Other identities are dimension specific. Let $\rho$ denote the Ricci tensor. For example, if
$m=3$, then we have the following identity (see Remark~\ref{rmk-1.2} for further details):
\begin{equation}\label{eqn-1.d}
\tau(A)^2-4|\rho(A)|^2+|A|^2=0\,.
\end{equation}
Since $\tau(A)^2$, $|\rho(A)|^2$, and $|A|^2$ are scalar invariants, this relation can be expressed in the
form $P_{2\ell,\mathcal{C},\varepsilon}$ as follows. We have
\begin{eqnarray*}
&&\tau(A)^2=\varepsilon^{i_1i_2}\varepsilon^{i_3i_4}\varepsilon^{i_5i_6}\varepsilon^{i_7i_8}
A_{i_1i_3i_4i_2}A_{i_5i_7i_8i_6},\\
&&|\rho(A)|^2=\varepsilon^{i_1i_2}\varepsilon^{i_3i_4}\varepsilon^{i_5i_6}\varepsilon^{i_7i_8}
A_{i_1i_3i_4i_5}A_{i_2i_7i_8i_6},\\
&&|A|^2=\varepsilon^{i_1i_2}\varepsilon^{i_3i_4}\varepsilon^{i_5i_6}\varepsilon^{i_7i_8}
A_{i_1i_3i_5i_7}A_{i_2i_4i_6i_8}\,.
\end{eqnarray*}
Thus the relation in Equation~(\ref{eqn-1.d}) can be defined by taking
$$c_\sigma=\left\{\begin{array}{rl}
1&\text{ if }\sigma=\left(\begin{array}{llllllll}12345678\\13425786\end{array}\right)\\
-4&\text{ if }\sigma=\left(\begin{array}{llllllll}12345678\\13452786\end{array}\right)
_{\vphantom{A}}^{\vphantom{A}}\\
1&\text{ if }\sigma=\left(\begin{array}{llllllll}12345678\\13572468\end{array}\right)
_{\vphantom{A}}^{\vphantom{A}}\\
0&\text{ otherwise}
\end{array}\right\}\,.$$
 We note that
this relation does not hold in dimension
$m=4$; fixing the underlying dimension of the vector space can be crucial.

\subsection{Changing the signature} The identity of Equation~(\ref{eqn-1.d}) is not specific to the
signature; it holds for any 3-dimensional pseudo-Riemannian manifold or curvature module -- i.e. in
signatures
$(0,3)$,
$(1,2)$,
$(2,1)$, and
$(3,0)$. More generally, the signature plays no role when considering universal curvature identities.
We will establish the following result subsequently in Section~\ref{sect-4}:

\begin{theorem}\label{thm-1.3}
 Let $\mathcal{C}:=\{c_\sigma\}_{\sigma\in\operatorname{Perm}(4\ell)}$
be a collection of real constants. Let $(V_i,\varepsilon_i)$ be inner product spaces
of signature $(p_i,q_i)$ where $m=p_1+q_1=p_2+q_2$. Then
$$
P_{2\ell,\mathcal{C},\varepsilon_1}(A_1)=0\ \forall\ A_1\in\mathfrak{A}(V_1)
\Leftrightarrow
P_{2\ell,\mathcal{C},\varepsilon_2}(A_2)=0\ \forall\ A_2\in\mathfrak{A}(V_2)\,.
$$
\end{theorem}

\begin{remark}\label{rmk-1.1}
\rm Let $P_{2\ell,\varepsilon}\in\mathcal{J}_{2\ell}(V,\varepsilon)$ be a polynomial invariant of degree $2\ell$ which is
defined for an inner product space $(V,\varepsilon)$ of signature $(p,q)$ in dimension $m=p+q$. By
Theorem~\ref{thm-1.2}, there is a collection of constants $\mathcal{C}$ so that may express:
$$
   P_{2\ell,\varepsilon}(A)=\sum_{\sigma\in\operatorname{Perm}(4\ell)} c_\sigma P_{2\ell,\sigma,\varepsilon}(A)
\ \forall\ A\in\mathfrak{A}(V)\,.
$$
If $(V_1,\varepsilon_1)$ is an inner product space of signature $(p_1,q_1)$ where $p_1+q_1=m$, we let
$$P_{2\ell,\varepsilon_1}(A_1):=\sum_{\sigma\in\operatorname{Perm}(4\ell)} c_\sigma P_{2\ell,\sigma,\varepsilon_1}(A_1)
\ \forall\ A_1\in\mathfrak{A}(V_1)\,.$$
If we choose another collection of constants $\tilde{\mathcal{C}}$ so that
$$P_{2\ell,\varepsilon}(A)=\sum_{\sigma\in\operatorname{Perm}(4\ell)} \tilde c_\sigma P_{2\ell,\sigma,\varepsilon}(A)
\ \forall\ A\in\mathfrak{A}(V)\,,$$
then
$$\sum_{\sigma\in\operatorname{Perm}(4\ell)}(c_\sigma-\tilde c_\sigma)P_{2\ell,\sigma,\varepsilon}(A)=0
\ \forall\ A\in\mathfrak{A}(V)\,.$$
Hence, by Theorem~\ref{thm-1.3},$$\sum_{\sigma\in\operatorname{Perm}(4\ell)}(c_\sigma-\tilde
c_\sigma)P_{2\ell,\sigma,\varepsilon_1}(A_1)=0\ \forall\
A_1\in\mathfrak{A}(V_1)\,.$$Thus we may conclude that we may also express
$$
P_{2\ell,\varepsilon_1}(A_1)=\sum_{\sigma\in\operatorname{Perm}(4\ell)}\tilde c_\sigma
P_{2\ell,\sigma,\varepsilon_1}(A_1)\ \forall\ A_1\in\mathfrak{A}(V_1)\,.
$$
This shows that we can
regard the collection $\{P_{2\ell,\varepsilon}\}$ as being defined for any inner product space of dimension $m$;
we shall denote this space of invariants by $\mathcal{J}_{2\ell}(m)$. The elements of $\mathcal{J}_{2\ell}(m)$
are functions; the elements $P_{2\ell,\sigma,\varepsilon}$ are algebraic objects which define functions, but
(as noted above) different elements $P_{2\ell,\sigma,\varepsilon}$ can define the same function.\end{remark}

\subsection{The restriction of scalar invariants}\label{sect-1.5} We wish to relate the spaces $\mathcal{J}_{2\ell}(m)$ and
$\mathcal{J}_{2\ell}(m-1)$ by defining a restriction map
$r:\mathcal{J}_{2\ell}(m)\rightarrow\mathcal{J}_{2\ell}(m-1)$. We work in the geometric context for the moment;
Theorem~\ref{thm-1.1} permits us to then pass to the algebraic context. Let
$P\in\mathcal{J}_{2\ell}(p,q)$ where $p>0$ and let $(N,g_N)$ be a pseudo-Riemannian manifold of signature $(p-1,q)$. Let
$M:=N\times S^1$ and let $g_M:=g_N-d\theta^2$ where $\theta$ is the usual periodic parameter on the circle $S^1$.Then $(M,g_M)$ has signature $(p,q)$ and we set$$r(P)(N,g_N)(\xi):=P(M,g_M)(\xi,\theta_0)$$
for any $\theta_0\in S^1$, the particular point $\theta_0$ being irrelevant as the circle is a homogeneous space.
This defines a map
$$
r_-:\mathcal{J}_{2\ell}(p,q)\rightarrow\mathcal{J}_{2\ell}(p-1,q)\quad\text{for}\quad p>0\,.
$$
If $q>0$ and if $(N,g_N)$ has signature $(p,q-1)$, we may consider $g_M:=g_N+d\theta^2$ to define a similar restriction
map
$$
r_+:\mathcal{J}_{2\ell}(p,q)\rightarrow\mathcal{J}_{2\ell}(p,q-1)\quad\text{for}\quad q>0\,.
$$

Let $\varepsilon(p,q)$ have signature $(p,q)$. The invariants $P_{2\ell,\sigma,\varepsilon(p,q)}$ are defined by
summations in Equation~(\ref{eqn-1.b}) that range from $1$ to $m$. The product metric $g_N\pm d\theta^2$ is flat in the
final coordinate and thus $r_\pm P_{2\ell,\sigma,\varepsilon(p,q)}$ is defined by summations which range from $1$ to
$m-1$. Consequently
\begin{eqnarray*}
&&r_-(P_{2\ell,\sigma,\varepsilon(p,q))})=P_{2\ell,\sigma,\varepsilon(p-1,q)}\quad\text{if}\quad p>0\,,\\
&&r_+(P_{2\ell,\sigma,\varepsilon(p,q))})=P_{2\ell,\sigma,\varepsilon(p,q-1)}\quad\text{if}\quad q>0\,.
\end{eqnarray*}
Consequently, we can regard $r_\pm$ as defining a unified and universally defined map
$$r:\mathcal{J}_{2\ell}(m)\rightarrow\mathcal{J}_{2\ell}(m-1)\,.$$
Thus, for example, the scalar curvature in
dimension $m$ restricts naturally to the scalar curvature in dimension $m-1$; it is universally
defined. It is necessary to first give a geometric definition and then invoke Theorem~\ref{thm-1.1} to ensure that
the subsequent algebraic characterization is well defined by showing
\begin{eqnarray*}
&&\sum_{\sigma\in\operatorname{Perm}(4\ell)}c_\sigma P_{2\ell,\sigma,\varepsilon_m}(A_m)=0\ \forall\
A_m\in\mathfrak{A}(V_m)\\
&\Rightarrow&\sum_{\sigma\in\operatorname{Perm}(4\ell)}c_\sigma P_{2\ell,\sigma,\varepsilon_{m-1}}(A_{m-1})=0\ \forall\
A_{m-1}\in\mathfrak{A}(V_{m-1})
\end{eqnarray*}
where $(V_m,\varepsilon_m)$ and $(V_{m-1},\varepsilon_{m-1})$ are arbitrary inner product spaces of dimensions $m$ and
$m-1$, respectively.

The following result
follows from the discussion given above:
\begin{lemma}
Let $\mathcal{C}=\{c_\sigma\}_{\sigma\in\operatorname{Perm}(4\ell)}$ be a collection of real constants which defines an
element $P_{2\ell,\mathcal{C},m}\in\mathcal{J}_{2\ell}(m)$. Then $r(P_{2\ell,\mathcal{C},m})=P_{2\ell,\mathcal{C},m-1}$.
Furthermore, if
$P_{2\ell,\mathcal{C},m}$ is a universal curvature identity in dimension $m$, then $r(P_{2\ell,\mathcal{C},m})$ is a
universal curvature identity in dimension $m-1$.
\end{lemma}

\subsection{The Pfaffian}
Define $E_{2\ell, m,\varepsilon}\in\mathcal{J}_{2\ell}(V,\varepsilon)$ by
setting:
\begin{eqnarray*}
&&\ E_{2\ell, m, \varepsilon}:=\frac1{(8\pi)^\ell
\ell!}\sum_{i_1,...,i_\ell,j_1,...,j_\ell=1}^{{\
m}}{{A_{i_1i_2j_2j_1}...A_{i_{\ell-1}i_\ell j_\ell j_{\ell-1}}}}\\
&&\qquad\qquad\qquad\qquad\qquad\qquad\qquad\
\times\varepsilon^*(e^{i_1}\wedge...\wedge
e^{i_\ell},e^{j_1}\wedge...\wedge e^{j_\ell})
\end{eqnarray*}
 where by definition one sets:
$$\varepsilon^*(e^{i_1}\wedge...\wedge
e^{i_{\ell}},e^{j_1}\wedge...\wedge
e^{j_{\ell}}):=\det\left(\begin{array}{lll}
\varepsilon^*(e^{i_1},e^{j_1})&\dots&\varepsilon^*(e^{i_1},e^{{{j_\ell}}})\\
\dots&\dots&\dots\\
\varepsilon^*(e^{{i_\ell}},e^{j_1})&\dots&
     \varepsilon^*(e^{{i_\ell}},e^{{{j_\ell}}})\end{array}\right)\,.
$$
It is then immediate that
$$
   r(E_{2\ell,m,\varepsilon})=E_{2\ell,m-1,\varepsilon}\,.
$$
Thus, in particular, this invariant is universal and will be denoted
by $E_{2\ell}$ when no confusion is likely to result. For example,
\begin{equation}\label{eqn-1.e}
E_2(A)=\frac{\tau(A)}{4\pi}\quad\text{and}\quad
E_4(A)=\frac{\tau(A)^{2}-4|\rho(A)|^2+|A|^2}{32\pi^2}\,.
\end{equation}
\begin{theorem}\label{thm-1.4}
\ \begin{enumerate}
\item $r:\mathcal{J}_{2\ell}(m)\rightarrow\mathcal{J}_{2\ell}(m-1)$ is surjective for any $m$.
\item If $m>2\ell$, then $r:\mathcal{J}_{2\ell}(m)\rightarrow\mathcal{J}_{2\ell}(m-1)$ is injective
\item If $m=2\ell$, then
$\ker\left\{r:\mathcal{J}_{2\ell}(m)\rightarrow\mathcal{J}_{2\ell}(m-1)\right\}=\operatorname{Span}\{E_{2\ell,m}\}$.
\end{enumerate}
\end{theorem}

\begin{proof} The discussion above using Theorem~\ref{thm-1.2} and Theorem~\ref{thm-1.3} shows that it suffices to prove
this result in the positive definite setting; this was done previously in \cite{GPS11}.
\end{proof}

\begin{remark}\label{rmk-1.2}
\rm We have $E_4(A)=\frac1{32\pi^2}(\tau(A)^2-4|\rho(A)|^2+|A|^2)$; $E_4$
is non-zero on
$(S^4,g_0)$ where $g_0$ is the round metric on the unit sphere $S^4$ in $\mathbb{R}^5$.
This invariant vanishes identically on any $3$-dimension pseudo-Riemannian manifold and thus provides
the only quadratic universal curvature identity (module rescaling) in dimension 3 which does not hold in
dimension 4.
\end{remark}

\subsection{The Chern-Gauss-Bonnet Formula}
In addition to being (up to scaling) the only universal curvature identity in dimension $2\ell-1$
which is non-trivial in dimension $2\ell$, the invariants
$E_{2\ell}$ are the integrands of the Chern-Gauss-Bonnet formula. Let
$\chi(M)$ be the Euler-Poincar\'e characteristic of a compact manifold $M$; since $\chi(M)=0$ if
$m=\dim(M)$ is odd we may suppose that $m=2\ell$ is even. Let
$$\left|\operatorname{dvol}\right|(g):=\left|\det(g_{ij})\right|^{1/2}dx^1\cdot\cdot\cdot dx^m$$
denote the Riemannian element of volume. We refer to \cite{C44}
for the proof of the following result in the Riemannian (positive definite) setting and to \cite{C63}
for the general case:
\begin{theorem}\label{thm-1.5}
Let $(M,g)$ be a compact pseudo-Riemannian manifold of dimension
$2\ell$ with empty boundary. Then
$$\int_M E_{2\ell}(M,g)\left|\operatorname{dvol}\right|(g)=\chi(M)\,.$$
\end{theorem}

\subsection{Symmetric $2$-tensor valued invariants}
Let $S^2(V^*)$ denote the space of symmetric $2$-cotensors. Let
$\mathcal{J}_{2\ell}^{(2)}(V,\varepsilon)$ denote the set of all $\mathcal{O}(V,\varepsilon)$ equivariant maps
from
$\otimes^\ell(\mathfrak{A}(V))$ to $S^2(V^*)$ or, equivalently, polynomials of degree $2\ell$ in the
components of $A\in\mathfrak{A}(V)$ which are $S^2(V^*)$ valued and invariantly defined. We can extend
Theorem~\ref{thm-1.2}, Theorem~\ref{thm-1.3}, and Theorem~\ref{thm-1.4} to this setting. We adopt the
following notational conventions. There are two fundamental invariants. If
$\eta\in\otimes^{4\ell}(V^*)$, define $Q_{2\ell,1,\varepsilon}^{(2)}\in\mathcal{J}_{2\ell}^{(2)}(V,\varepsilon)$ and
$Q_{2\ell,2,\varepsilon}^{(2)}\in\mathcal{J}_{2\ell}^{(2)}(V,\varepsilon)$ by setting:\begin{eqnarray*}
&&Q_{2\ell,1,\varepsilon}^{(2)}(\eta):=\varepsilon^{i_1i_2}...\varepsilon^{i_{4\ell-3}i_{4\ell-2}}\eta_{i_1i_2...i_{4\ell-1}i_{4\ell}}e^{i_{4\ell-1}}\circ e^{i_{4\ell}},\\&&Q_{2\ell,2,\varepsilon}^{(2)}(\eta):=\varepsilon^{i_1i_2}...\varepsilon^{i_{4\ell-1}i_{4\ell}}
\eta_{i_1i_2...i_{4\ell-1}i_{4\ell}}\varepsilon^*\,.
\end{eqnarray*}
If $\sigma\in\operatorname{Perm}(4\ell)$, then as before, we shall define
$$
\begin{array}{ll}
Q_{2\ell,1,\sigma,\varepsilon}^{(2)}:=Q_{2\ell,1,\varepsilon}^{(2)}\circ\sigma,&\left.
P_{2\ell,1,\sigma,\varepsilon}^{(2)}:=Q_{2\ell,1,\sigma,\varepsilon}^{(2)}\right|_{\otimes^\ell(\mathfrak{A}(V))},\\
Q_{2\ell,2,\sigma,\varepsilon}^{(2)}:=Q_{2\ell,2,\varepsilon}^{(2)}\circ\sigma,&\left.
P_{2\ell,1,\sigma,\varepsilon}^{(2)}:=Q_{2\ell,1,\sigma,\varepsilon}^{(2)}\right|_{\otimes^\ell(\mathfrak{A}(V))}.
\vphantom{\vrule height 16pt}\end{array}
$$
We will establish the following extension of Theorem~\ref{thm-1.2} in
Section~\ref{sect-3}:

\begin{theorem}\label{thm-1.6}
$\mathcal{J}_{2\ell}^{(2)}(V,\varepsilon)=
\operatorname{Span}_{\sigma\in\operatorname{Perm}(4\ell)}
\left\{P^{(2)}_{2\ell,1,\sigma,\varepsilon},P^{(2)}_{2\ell,2,\sigma,\varepsilon}\right\}$.
\end{theorem}

If $
\mathcal{D}:=\{d_{1,\sigma}^{(2)},d_{2,\sigma}^{(2)}\}_{\sigma\in\operatorname{Perm}(4\ell)}
$
is a collection of real
constants, we set
\begin{equation}\label{eqn-1.f}
P_{2\ell,\mathcal{D},\varepsilon}^{(2)}:=\sum_{\sigma\in\operatorname{Perm}(4\ell)}\{d_{1,\sigma}^{(2)}P_{2\ell,1,\sigma,\varepsilon}^{(2)}+d_{2,\sigma}^{(2)}P_{2\ell,2,\sigma,\varepsilon}^{(2)}\}\end{equation}

We shall establish the following extension of
Theorem~\ref{thm-1.3} in
Section~\ref{sect-4}:

\begin{theorem}\label{thm-1.7}
Let $\mathcal{D}:=\{d_{1,\sigma}^{(2)},d_{2,\sigma}^{(2)}\}_{\sigma\in\operatorname{Perm}(4\ell)}$
be a collection of real constants. Let $(V_i,\varepsilon_i)$ be inner product spaces
of the same dimension $m=p_1+q_1=p_2+q_2$. Then
$$P_{2\ell,\mathcal{D},\varepsilon_1}^{(2)}(A_1)=0\ \forall\ A_1\in\mathfrak{A}(V_1)
\quad\Leftrightarrow\quad
P_{2\ell,\mathcal{D},\varepsilon_2}^{(2)}(A_2)=0\ \forall\ A_2\in\mathfrak{A}(V_2)\,.$$
\end{theorem}

As in the scalar case, we use Theorem~\ref{thm-1.6} and Theorem~\ref{thm-1.7} to identify
$\mathcal{J}_{2\ell}^{(2)}(p,q)$ with
$\mathcal{J}_{2\ell}^{(2)}(p_1,q_1)$ if $p+q=p_1+q_1=m$ and to define a universal space of invariants
$\mathcal{J}_{2\ell}^{(2)}(m)$ in dimension $m$. The restriction maps
\begin{eqnarray*}
&&r_-^{(2)}:\mathcal{J}_{2\ell}^{(2)}(p,q)\rightarrow\mathcal{J}_{2\ell}^{(2)}(p-1,q)\quad\text{for}\quad p>0\,,\\
&&r_+^{(2)}:\mathcal{J}_{2\ell}^{(2)}(p,q)\rightarrow\mathcal{J}_{2\ell}^{(2)}(p,q-1)\quad\text{for}\quad q>0\,,
\end{eqnarray*}
are defined geometrically as before by setting:
$$
\{r_\pm^{(2)}(P_{2\ell}^{(2)})\}(N,g_N)(\xi)=i_N^*\{P_{2\ell}^{(2)}(N\times S^1,g_N\pm d\theta^2)(\xi,\theta_1)\}
$$
where $i_N^*$ is the dual map induced by the inclusion $N\rightarrow
N\times\{\theta_1\}\subset{ N\times S^1}$. The additional bit of
technical fuss defined in using $i_N^*$ is required as it is
necessary to restrict a symmetric $2$-cotensor from $N\times S^1$ to
$N$ (this is not necessary for scalar valued invariants). The
restriction map $r_+^{(2)}$ is defined similarly if $q>0$. As
before, the restriction maps patch together to define a coherent map
$r^{(2)}$ from $\mathcal{J}_{2\ell}^{(2)}(m)$ to
$\mathcal{J}_{2\ell}^{(2)}(m-1)$ so that
\begin{eqnarray*}
{{r_-^{(2)}}}(P_{2\ell,1,\sigma,\varepsilon(p,q)}^{(2)})=P_{2\ell,1,\sigma,\varepsilon(p-1,q)}^{(2)}\quad\text{if}\quad p>0\,,\\
{{r_+^{(2)}}}(P_{2\ell,2,\sigma,\varepsilon(p,q)}^{(2)})=P_{2\ell,2,\sigma,\varepsilon(p,q-1)}^{(2)}\quad\text{if}\quad
q>0\,.
\end{eqnarray*}
The reason to give a geometric definition first, of course, was to
ensure that the image of a universal curvature identity was again a
universal curvature identity so that
$r^{(2)}(P_{2\ell}^{(2)})$ was defined independent of the
representation of $P_{2\ell}^{(2)}$ in terms of the fundamental
invariants
$\{P_{2\ell,1,\sigma,\varepsilon},P_{2\ell,2,\sigma,\varepsilon}\}$.

In analogy with the Pfaffian, we define ${
T_{2\ell,m,\varepsilon}^{(2)}}\in\mathcal{J}_{2\ell}^{(2)}(V,\varepsilon)$
by setting: {$$\begin{array}{l} \displaystyle
T_{2\ell,m,\varepsilon}^{(2)}:=\sum_{i_1,...,i_{\ell+1},j_1,...,j_{\ell+1}=1}^{m}A_{i_1i_2j_2j_1}...
A_{i_{\ell-1}i_\ell j_\ell j_{\ell-1}}e^{i_{\ell+1}}\circ
e^{j_{\ell+1}}\\
\qquad\qquad\qquad\qquad\times
\varepsilon^*(e^{i_1}\wedge...\wedge
e^{i_{\ell+1}},e^{j_1}\wedge...\wedge e^{j_{\ell+1}})\,.\vphantom{\vrule
height 11pt}
\end{array}$$}
We let $T_{2\ell,m}^{(2)}$ denote the invariants $\left\{{
T_{2\ell,m,\varepsilon}^{(2)}}\right\}$ in dimension $m$. We then
have that
$$r^{(2)}(T_{2\ell,m}^{(2)})=T_{2\ell,m-1}^{(2)}\,.$$
Consequently, once again, these invariants are universal. Theorem~\ref{thm-1.4} generalizes to this setting to become:

\begin{theorem}\label{thm-1.8}
\ \begin{enumerate}
\item $r^{(2)}:\mathcal{J}_{2\ell}^{(2)}(m)\rightarrow\mathcal{J}_{2\ell}^{(2)}(m-1)$ is surjective for any $m$.
\smallbreak\item If $m>2\ell+1$, then
$r^{(2)}:\mathcal{J}_{2\ell}^{(2)}(m)\rightarrow\mathcal{J}_{2\ell}^{(2)}(m-1)$
is injective. \smallbreak\item If $m=2\ell+1$, then
$\ker\left\{r^{(2)}:\mathcal{J}_{2\ell}^{(2)}(m)\rightarrow\mathcal{J}_{2\ell}^{(2)}(m-1)\right\}
=\operatorname{Span}\left\{T_{2\ell,m}^{(2)}\right\}$.
\end{enumerate}
\end{theorem}

\begin{proof} The discussion above using Theorem~\ref{thm-1.6} and Theorem~\ref{thm-1.7} shows that it suffices to prove
Theorem~\ref{thm-1.8} in the positive definite context. This was done previously in \cite{GPS11}.
\end{proof}

\begin{remark}\rm If $\ell=2$, we define $Q_{4,\varepsilon,m}^{(2)}\in\mathcal{J}_4^{(2)}(m)$ by setting:
\begin{eqnarray*}
&&Q_{4,\varepsilon,m}^{(2)}(A):=-\textstyle\frac14\varepsilon^{i_1i_2}\varepsilon^{i_3i_4}
\varepsilon^{i_5i_6}\varepsilon^{i_7i_8}\varepsilon_{i_9i_{10}}\\
&&\qquad\times\left(A_{i_1i_3i_4i_2}A_{i_5i_7i_8i_6}
   -4A_{i_1i_3i_4i_5}A_{i_2i_7i_8i_6}
+A_{i_1i_3i_5i_7}A_{i_2i_4i_6i_8}\right)e^{i_9}\circ
e^{i_{10}}\\
&&\quad+\delta_{i_2}^{i_1}\delta_{i_4}^{i_3}
\varepsilon^{i_5i_6}\varepsilon^{i_7i_8}\varepsilon^{i_9i_{10}}\\
&&\qquad\times\left(A_{i_5i_7i_9i_1}A_{i_6i_8i_{10}i_3}-2A_{i_5i_9i_1i_6}A_{i_7i_{10}i_3i_8}\right)e^{i_2}\circ
e^{i_4}\\ &&\quad+\delta_{i_2}^{i_1}\delta_{i_4}^{i_3}
\varepsilon^{i_5i_6}\varepsilon^{i_7i_8}\varepsilon^{i_9i_{10}}\\
&&\qquad\times\left(-2A_{i_1i_5i_7i_3}A_{i_9i_6i_8i_{10}}+A_{i_5i_7i_8i_6}A_{i_9i_1i_3i_{10}}\right)e^{i_2}\circ
e^{i_4}\,.
\end{eqnarray*}
Euh, Park, and Sekigawa \cite{EPS10} showed that if
$\varepsilon$ is positive definite and if $m=4$, then
$$Q_{4,\varepsilon,4}^{(2)}(A)=0\quad\forall\quad A\in\mathfrak{A}(V)\,.$$
Note that this does not hold true if $m\ge5$ so this is a universal curvature identity which holds in
dimension 4 but not in dimension 5. Theorem~\ref{thm-1.7} shows this relation holds in signature $(0,4)$,
$(1,3)$, $(2,2)$, and $(1,3)$ as well. By Theorem~\ref{thm-1.8}, there is a universal constant
$c$ so that $Q_{4,\varepsilon,4}^{(2)}=cT_{4,4}^{(2)}$. We refer to
\cite{EJP11} for an evaluation of the normalizing constants which arise when this invariant is expressed in terms of a
Weyl basis using the methods of universal examples.
\end{remark}

\subsection{Euler-Lagrange Equations}\label{sect-1.9} Let $h$ be an arbitrary symmetric $2$-cotensor field on a compact
pseudo-Riemannian manifold $(M,g)$ of dimension $m$. We form the $1$-parameter family
$g(\vartheta):=g+\vartheta h$; this is non-degenerate for $\vartheta$ small. Since the Pfaffian $E_{2\ell}$ only involves
the first and second derivatives of the metric, the variation only
involves the first and second derivatives of $h$. We may therefore
express:
\begin{eqnarray*}
&&\left.\partial_\vartheta\left\{
\vphantom{\vrule height 11pt}E_{2\ell,m}(g(\vartheta))\left|\operatorname{dvol}\right|(g(\vartheta))\right\}
\right|_{\vartheta=0}\\
&=&\left\{\vphantom{\vrule height 11pt}
Q_{2\ell,m}^{ij}h_{ij}+Q_{2\ell,m}^{ijk}h_{ij;k}+Q_{2\ell,m}^{ijkl}h_{ij;kl}\right\}\left|\operatorname{dvol}\right|(g)\end{eqnarray*}\medbreak\noindent
where $h_{ij;k}$ and $h_{ij;kl}$ give the components of the
covariant derivative of $h$ with respect to the Levi-Civita
connection of $g$.  Let $Q_{2\ell,m;l}^{ijk}=$ and
$Q_{2\ell,m;uv}^{ijkl}$ be the components of the first and second
covariant derivatives of these tensors, respectively.  Define:
\begin{equation}\label{eqn-1.g}
S_{2\ell,m}^{ij}:=Q_{2\ell,m}^{ij}-Q_{2\ell,m;k}^{ijk}+Q_{2\ell,m;lk}^{ijkl}\,.
\end{equation}
The tensor $S^{ij}_{2\ell,m}e_i\circ e_j\in S^2(V)$ is characterized by the property that if $(M,g)$ is any
compact pseudo-Riemannian manifold of dimension $m$ with empty boundary, then we may integrate
by parts to see that:
$$\partial_\vartheta\left.\left\{\int_ME_{2\ell,m}(g(\vartheta))\left|\operatorname{dvol}\right|_{g(\vartheta)}\right\}
\right|_{\vartheta=0}
=\int_MS_{2\ell,m}^{ij}h_{ij}\left|\operatorname{dvol}\right|(g)\,.
$$
We use the metric to raise and lower indices and to define the corresponding symmetric
2-cotensor $S_{2\ell,m}^{(2)}$ by the identity:
$$S^{ij}_{2\ell,m}h_{ij}=g(S_{2\ell,m}^{(2)},h)\,.$$
By Theorem~\ref{thm-1.5}, this vanishes if
$m=2\ell$.  In Section~\ref{sect-5}, we shall establish a conjecture of Berger \cite{M70} in the
pseudo-Riemannian setting:
\begin{theorem}\label{thm-1.9}
$S_{2\ell,m}^{(2)}=c_{2\ell} T_{2\ell,m}^{(2)}\in\mathcal{J}_{2\ell,m}^{(2)}$ for any $(\ell,m)$.\end{theorem}

A-priori, Equation~(\ref{eqn-1.g}) involves the second covariant
derivatives of the curvature tensor. There are appropriate
restriction maps in this category. As
$$
    {{r_\pm^{(2)}}}(S_{2\ell,m}^{(2)})=S_{2\ell,m-1}^{(2)}\,,
$$
these
invariants are universally defined. We will establish Theorem~\ref{thm-1.9} by showing that in fact
$S_{2\ell,m}^{(2)}$ only depends on the curvature tensor and not on its covariant derivatives and thus
$S_{2\ell,m}^{(2)}\in\mathcal{J}_{2\ell,m}^{(2)}$. By Theorem~\ref{thm-1.5}, the Euler-Lagrange
equations vanish if $m=2\ell$ and thus $S_{2\ell,m}^{(2)}$ belongs to $\ker(r_\pm)$. We may then use
Theorem~\ref{thm-1.8} to established the desired identity
$$S_{2\ell,m}^{(2)}=c_{2\ell}T_{2\ell,m}^{(2)}\,.$$

 In the Riemannian setting, this result is not new. It was first established
by Kuz'mina \cite{K74} and subsequently established using different methods by Labbi \cite{L05,L07,L08}.
The present paper was motivated in part by a desire to extend this result from the Riemannian setting
to the pseudo-Riemannian setting.

\section{The proof of Theorem~\ref{thm-1.1}}\label{sect-2} Let
$(V,\varepsilon)$ be an inner product space and let $A\in\mathfrak{A}(V)$. Fix a
basis $\{e_i\}$ for $V$ to identify $V$ with $\mathbb{R}^m$. Define:
$$
\textstyle g_{ik}:=\varepsilon_{ik}-\frac13A_{ijlk}x^jx^l\,.
$$
Clearly $g_{ik}=g_{ki}$. As $g_{ik}(0)=\varepsilon_{ik}$ is non-singular,  $g$ is a
pseudo-Riemannian metric on some neighborhood of the origin in $\mathbb{R}^m$.
Let
$$g_{ij/k}:=\partial_{x_k}g_{ij}\quad\text{and}\quad
  g_{ij/kl}:=\partial_{x_k}\partial_{x_l}g_{ij}\,.
$$
The Christoffel symbols of the first kind\index{Christoffel symbols of the first kind} are given by:
$$
\Gamma_{ijk}:=g(\nabla_{\partial_{x_i}}\partial_{x_j},\partial_{x_k})
       =\textstyle\frac12(g_{jk/i}+g_{ik/j}-g_{ij/k})\,.
$$
As $g=\varepsilon+O(|x|^2)$ and $\Gamma=O(|x|)$, we may compute:
\medbreak\qquad
$R_{ijkl}\phantom{.}=\{\partial_{x_i}\Gamma_{jkl}-\partial_{x_j}\Gamma_{ikl}\}+O(|x|^2)$
\smallbreak\qquad\qquad\,\,
$=\textstyle\frac12\{g_{jl/ik}+g_{ik/jl}
    -g_{jk/il}-g_{il/jk}
    \}+O(|x|^2)$
\smallbreak\qquad\qquad\,\,
$=
 \textstyle\frac16\{-A_{jikl}-A_{jkil}-A_{ijlk}-A_{iljk}$
\smallbreak\qquad\qquad\qquad
$   +A_{jilk}+A_{jlik}+A_{ijkl}+A_{ikjl}\}+O(|x|^2)$
\smallbreak\qquad\qquad\,\,
$=\textstyle\frac16\{4A_{ijkl}-2A_{iljk}-2A_{iklj}\}+O(|x|^2)$
\smallbreak\qquad\qquad\,\,
$ =A_{ijkl}+O(|x|^2)$.
\medbreak This argument shows that $(\mathbb{R}^m,g)$ is the germ of a
pseudo-Riemannian manifold which provides a geometric realization of the model
$(V,\varepsilon,A)$ at the origin; it is non-degenerate on some neighborhood $\mathcal{U}$ of the
origin. Let
$\mathbb{T}^m:=\mathbb{R}^m/\mathbb{Z}^m$ be the torus and let $(\theta^1,...,\theta^m)$ be the usual
periodic parameters. We define the flat metric
$h(\partial_{\theta_i},\partial_{\theta_j})=\varepsilon_{ij}$. We may regard
$\mathcal{U}$ as a neighborhood of $0$ in $\mathbb{T}^m$ as well. Let
$\phi$ be a plateau function which is identically 1 near $0$ in $\mathbb{T}^m$
and which has compact support in $\mathcal{U}$. We consider $g^\phi:=\phi
g+(1-\phi)h$. Since $g^\phi$ agrees with $g$ on a smaller neighborhood of $0$,
$(\mathbb{T}^m,g)$ provides a geometrical realization of $(V,\varepsilon,A)$ at
$0$. Since $g(0)=h(0)$, $g^\phi$ is non-degenerate for $\phi$ with sufficiently
small support. \hfill\qed

\begin{remark}\rm The role of the torus is inessential; any pseudo-Riemannian
manifold of the proper signature could have been used; the crucial point is that the manifold in question
should admit some background pseudo-Riemannian manifold of the given signature.
\end{remark}

\section{The proof of Theorem~\ref{thm-1.2} and of Theorem~\ref{thm-1.6}}\label{sect-3}
Let $(V,\varepsilon)$ be an inner product space and let $(V^*,\varepsilon^*)$ be the dual
inner product space. If $\vec v=(v^1,\dots,v^k)$ and if $\vec w=(w^1,\dots,w^k)$
are elements of $\times^k(V^*)$, the map
$$\vec v\times\vec w\rightarrow\varepsilon^*(v^1,w^1)\cdot\cdot\cdot
\varepsilon^*(v^k,w^k)$$ is a bilinear symmetric map from
{{$\left\{\times^k V^*\right\}\times\left\{\times^k V^*\right\}$ to
$\mathbb{R}$}} which extends to a symmetric inner
product\index{inner product} on $\varepsilon^*$ to $\otimes^k(V^*)$.
If $\{e_i\}$ is an orthonormal basis for $V$, let $\{e^i\}$ be the
corresponding orthonormal basis for $V^*$. If $I=(i_1,\dots,i_k)$ is
a multi-index, let $e^I:=e^{i_1}\otimes\dots\otimes e^{i_k}$. The
collection $\{e^I\}_{|I|=k}$ forms a basis for $\otimes^k(V^*)$ with
$$\varepsilon^*( e^I,e^K)=\left\{\begin{array}{rll}0&\quad\text{if}&I\ne K\\
\varepsilon^*(e^{i_1},e^{i_1})\dots\varepsilon^*
(e^{i_k},e^{i_k})&\quad\text{if}&I=K\end{array}\right\}.$$
Since
$\varepsilon^*(e^I,e^I)=\pm1$, $\varepsilon^*$ is non-degenerate on $\otimes^k(V^*)$. The
orthogonal group\index{orthogonal group}
$\mathcal{O}(V,\varepsilon)$ extends to act naturally on
$\otimes^k(V^*)$ by pull-back and preserves this inner product\index{inner product}.
We have the following useful, if elementary, observation:

\begin{lemma}\label{lem-3.1}
Let $W$ be an $\mathcal{O}(V,\varepsilon)$ invariant subspace of $\otimes^k(V^*)$.
Then the restriction of $\varepsilon^*$ to $W$ is non-degenerate.
\end{lemma}

\begin{proof} Find an orthogonal direct sum decomposition $V=V_+\oplus V_-$
where $V_+$ is
spacelike and $V_-$ is timelike, i.e. the restriction of the inner product to $V_+$ is positive
definite and the restriction of the inner product to $V_-$ is negative definite. Let
$\Theta=\pm\operatorname{Id}$ on
$V_\pm$ define an element of
$\mathcal{O}(V,\varepsilon)$. Let
$\{e_1,\dots,e_p\}$ be an orthonormal basis for $V_-$ and let $\{e_{p+1},\dots,e_m\}$ be an
orthonormal basis for $V_+$. We have that:
$$
\Theta^*e^I=\left\{\varepsilon^*(e^{i_1},e^{i_1})\cdot\cdot\cdot\varepsilon^*(e^{i_k},e^{i_k})\right\}
e^I=\varepsilon^*(e^I,e^I)e^I=\pm e^I\,.
$$
Let $0\ne w $. If
$\Theta^*w=w$, then
$w$ is a spacelike vector in
$\otimes^k(V^*)$ (i.e. $\varepsilon^*(w,w)>0$) while if $\Theta^*w=-w$, then $w$ is a timelike vector in
$\otimes^k(V^*)$ (i.e. $\varepsilon^*(w,w)<0$). Thus, in particular, the induced inner product on
$\otimes^k(V^*)$ is non-degenerate.

Let $W$ be a non-trivial $\mathcal{O}(V,\varepsilon)$ invariant subspace of $\otimes^k(V^*)$. Since
$\Theta\in\mathcal{O}(V,\varepsilon)$, $\Theta$ preserves
$W$ by assumption. Thus we may decompose $W=W_+\oplus W_-$ into the $\pm1$ eigenspaces of $\Theta$. Since $\Theta$ acts
orthogonally, $W_+\perp W_-$. Since
$W_+$ is spacelike and
$W_-$ is timelike,
the induced inner product\index{inner product} on $W$ is non-degenerate.\end{proof}
\subsection{The proof of Theorem~\ref{thm-1.2}} Let $P_{2\ell}$ be an invariant map from$\otimes^\ell(\mathfrak{A}(V))$ to $\mathbb{R}$. By Lemma~\ref{lem-3.1}, we have an equivariant
orthogonal decomposition
$$\otimes^{4\ell}(V^*)=\left\{\otimes^\ell(\mathfrak{A}(V))\right\}
     \oplus\left\{\otimes^\ell(\mathfrak{A}(V))\right\}^\perp\,.$$
Extend $P_{2\ell}$ to be $0$ on $\{\otimes^\ell(\mathfrak{A}(V))\}^\perp$ to define an invariant map $Q_{2\ell}$
from $\otimes^{4\ell}(V^*)$ to $\mathbb{R}$ which restricts to yield $P_{2\ell}$ on
$\otimes^\ell(\mathfrak{A}(V))$; there are, of course, many possible such extensions. We may now use H.
Weyl's theorem
\cite{W46} concerning the invariants of the orthogonal group to express $Q_{2\ell}$ in terms of the invariants
$\{Q_{2\ell,\sigma,\varepsilon}\}$; restricting once again to $\otimes^\ell(\mathfrak{A}(V))$ then establishes
Theorem~\ref{thm-1.2} by expressing the original invariant $P_{2\ell}$ in terms of the invariants
$\{P_{2\ell,\sigma,\varepsilon}\}$.\hfill\qed

\subsection{The proof of Theorem~\ref{thm-1.6}}
Let $P_{2\ell}^{(2)}$ be an $\mathcal{O}(V,\varepsilon)$ map from $\otimes^\ell(\mathfrak{A}(V))$ to
$S^2(V^*)$. Contracting with $\varepsilon^*$ then defines an invariant map
$$\eta^{(2)}:\otimes^\ell(\mathfrak{A}(V))\otimes S^2(V^*)\rightarrow\mathbb{R}\,.$$
The subspace $\otimes^\ell(\mathfrak{A}(V))\otimes S^2(V^*)$ is an
$\mathcal{O}(V,\varepsilon)$invariant subspace of
$\otimes^{4\ell+2}(V^*)$. Consequently we may argue as above and use
Lemma~\ref{lem-3.1} to extend $\eta^{(2)}$ to an invariant map
$\tilde\eta^{(2)}$ from $\otimes^{4\ell+2}(V^*)$ to $\mathbb{R}$.
Again, H. Weyl's theorem shows that $\tilde\eta^{(2)}$ can be
expressed in terms of contractions of indices. Thus using the
notation of Equation~(\ref{eqn-1.c}), we can express $\eta^{(2)}$ as
a sum of invariants of the form
$$
\eta^{(2)}(A,h)=\varepsilon^*(P_{2\ell}^{(2)}(A),h)=\sum_{\beta\in\operatorname{Perm}(4\ell+2)}d_{\beta}
Q_{2\ell,\beta,\varepsilon}(A,h)
$$
where
\begin{eqnarray*}
Q_{2\ell,\beta,\varepsilon}(A,h)&:=&
\varepsilon^{i_{\beta_1}i_{\beta_2}}\cdot\cdot\cdot
\varepsilon^{i_{\beta_{4\ell+1}}i_{\beta_{4\ell+2}}}\\&&\quad\times
A_{i_1i_2i_3i_4}\cdot\cdot\cdot A_{i_{4\ell-3}i_{4\ell-2}i_{4\ell-1}i_{4\ell}}h_{i_{4\ell+1}i_{4\ell+2}}
\,.\end{eqnarray*}
We distinguish 2 cases:
\begin{enumerate}
\item Suppose that
$\{\beta_{2j-1},\beta_{2j}\}=\{{4\ell+1},{4\ell+2}\}$ for some
$j$. Then $h_{ij}$ is contracted against itself and all the indices of $A$ are contracted against
indices of $A$. By renumbering things, we may in fact assume that $\beta_{4\ell+1}=4\ell+1$ and
$\beta_{4\ell+2}=4\ell+2$ so that
$$\qquad\qquad Q_{2\ell,\beta,\varepsilon}(A,h)=
\varepsilon^{i_{\sigma_1}i_{\sigma_2}}...\varepsilon^{i_{\sigma_{4\ell-1}}i_{\sigma_{4\ell}}}
A_{i_1i_2i_3i_4}...A_{i_{4\ell-3}i_{4\ell-2}i_{4\ell-1}i_{4\ell}}\varepsilon^{ij}h_{ij}$$
for some $\sigma\in\operatorname{Perm}(4\ell)$. We then have
$$Q_{2\ell,\beta,\varepsilon}(A,h)=\varepsilon^*(P_{2\ell,2,\tilde\sigma,\varepsilon}^{(2)}(A),h)\quad\text{where}
\quad\tilde\sigma=\sigma^{-1}\,.$$
\item If $h_{ij}$ is contracted against indices of $A$, then
we obtain:
$$Q_{2\ell,\beta,\varepsilon}(A,h)=\varepsilon^*(P_{2\ell,1,\sigma,\varepsilon}^{(2)}(A),h)$$
for some suitably chosen $\sigma\in\operatorname{Perm}(4\ell)$.
\end{enumerate}

This analysis permits us to express:
$$\varepsilon^*(P_{2\ell}^{(2)}(A),h)=\varepsilon^*\left(\left\{\sum_{\sigma\in\operatorname{Perm}(4\ell)}
   (d_{1,\sigma,\varepsilon}^{(2)}P_{2\ell,1,\sigma,\varepsilon}^{(2)}(A)
    +d_{2,\sigma,\varepsilon}^{(2)}P_{2\ell,2,\sigma,\varepsilon}^{(2)}(A)\right\},h\right)\,.$$
Since this holds for all $h$, we conclude
\medbreak\hfill
$P_{2\ell}^{(2)}(A)=\sum_{\sigma\in\operatorname{Perm}(4\ell)}
   (d_{1,\sigma,\varepsilon}^{(2)}P_{2\ell,1,\sigma,\varepsilon}^{(2)}(A)
   +d_{2,\sigma,\varepsilon}^{(2)}P_{2\ell,2,\sigma,\varepsilon}^{(2)})(A)\,.$\hfill\qed

\section{The proof of Theorem~\ref{thm-1.3} and of Theorem~\ref{thm-1.7}}\label{sect-4}
\begin{proof} An inner product space is determined, up to isomorphism, by its
signature. Fix a positive definite inner product space
$(V,\varepsilon)$ of dimension
$m$. Theorem~\ref{thm-1.3} will follow if given any signature $(p,q)$ with
$p+q=m$, we can find an inner product space
$(V(p,q),\varepsilon(p,q))$ of signature $(p,q)$ so that
$P_{2\ell,\mathcal{C},\varepsilon}(A)=0$ for all
$A\in\mathfrak{A}(V)$ if and only
$P_{2\ell,\mathcal{C},\varepsilon(p,q)}(A)=0$ for all
$A\in\mathfrak{A}(V(p,q))$; Theorem~\ref{thm-1.7} will follow similarly.

We shall use analytic continuation. Let $W:=V\otimes_{\mathbb{R}}\mathbb{C}$ be
the complexification of $V$, and let
$\mathfrak{A}(W):=\mathfrak{A}(V)\otimes_{\mathbb{R}}\mathbb{C}$ be
the complexification of $\mathfrak{A}(V)$; $\mathfrak{A}(W)$ is the complex
vector space consisting of all elements of $\otimes^4(W^*)$ satisfying
Equation~(\ref{eqn-1.a}). Let $\varepsilon_W$ be the complexification of
$\varepsilon_V$; $\varepsilon_W$ is a symmetric complex bilinear form. Let
$\mathcal{O}_{\mathbb{C}}(W,\varepsilon_W)$ be the complex orthogonal group; this is the group of all
complex linear maps of $W$ preserving $\varepsilon_W$. This group acts
naturally on $\mathfrak{A}(W)$ by pullback. We use Equation~(\ref{eqn-1.b}) to
extend
$P_{2\ell,\mathcal{C},\varepsilon}$ to
$\mathfrak{A}(W)$ to be invariant under the structure group
$\mathcal{O}_{\mathbb{C}}(W,\varepsilon_W)$; this is independent of the particular complex basis chosen
for the complex vector space $W$.

We suppose given a signature $(p,q)$ with $p+q=m$. Let
$$
V(p,q):=\text{Span}_{\mathbb R}\{\sqrt{-1} e_1,
\cdot\cdot\cdot, \sqrt{-1} e_p, e_{p+1}, \cdot\cdot\cdot, e_{p+q}
\}
$$
and let $\varepsilon(p,q)$ be the restriction of $\varepsilon_W$ to $V(p,q)$;
$\varepsilon(p,q)$ is a real inner product of signature $(p,q)$ on the real vector space $V(p,q)$.
Note that
\begin{eqnarray*}
&&(V,\varepsilon)=(V{(0,m)},\varepsilon{(0,m)}),\quad
V({{0}},q)\otimes_{\mathbb{R}}{\mathbb{C}}=W,\\
&&\mathfrak{A}(V({{0}},q))\otimes_{\mathbb{R}}\mathbb{C}=\mathfrak{A}(W)\,.
\end{eqnarray*}

Let $U$ be a real vector space of dimension $r$ and let
$U_{\mathbb{C}}:=U\otimes_{\mathbb{R}}\mathbb{C}$. Suppose that
$P=P(u_1,...,u_r)$ is holomorphic on $\times^r(U_{\mathbb{C}})$. By the identity
theorem, one then has that
$P(u)=0$ for all
$u\in U$ if and only if $P(u)=0$ for all $u\in U_{\mathbb{C}}$. We apply this
observation as follows. Let $\mathcal{C}=
\{c_\sigma\}_{\sigma\in {\operatorname{Perm}}(4\ell)}$ be a collection of
complex constants. Then
$P_{2\ell,\mathcal{C},\varepsilon_W}$ is a holomorphic on $\mathfrak{A}(W)$.
Consequently, the following assertions are equivalent:
\begin{enumerate}
\item $P_{2\ell,\mathcal{C},\varepsilon}(A)=0$ for all $A\in\mathfrak{A}(V)$ $\Leftrightarrow$
\item $P_{2\ell,\mathcal{C},\varepsilon}(A)=0$ for all $A\in\mathfrak{A}(V)\otimes_{\mathbb{R}}\mathbb{C}$
 $\Leftrightarrow$
\item $P_{2\ell,\mathcal{C},\varepsilon(p,q)}(A)=0$ for all
$A\in\mathfrak{A}(V(p,q))\otimes_{\mathbb{R}}\mathbb{C}$
 $\Leftrightarrow$
\item $P_{2\ell,\mathcal{C},\varepsilon(p,q)}(A)=0$ for all
$A\in\mathfrak{A}(V(p,q))$.
\end{enumerate}
Theorem~\ref{thm-1.3} now follows; the proof of Theorem~\ref{thm-1.7} is similar and is therefore
omitted in the interests of brevity.
\end{proof}

\section{The proof of Theorem~\ref{thm-1.9}}\label{sect-5}
A-priori the invariant $S_{2\ell,m}^{(2)}$ of Section~\ref{sect-1.9} can involve first and second covariant derivatives of
the curvature tensor. Our first task is to see that this does not happen. We fix a signature $(p,q)$. For the moment, we
work with a coordinate based formalism. Let $\xi$ be a point of a pseudo-Riemannian manifold $(M,g)$ of signature $(p,q)$
and dimension $m=p+q$. Let
$X=(x^1,...,x^m)$ be a system of local coordinates on $M$ which are defined in a neighborhood of $\xi$ in $M$. Let
$$
g_{ij}(X,g):=g(\partial_{x_i},\partial_{x_j})
$$
give the components of the metric tensor. If $\alpha=(a_1,...,a_m)$, is a multi-index, set$$g_{ij/\alpha}(X,g):=\partial_{x_1}^{a_1}...\partial_{x_m}^{a_m}g_{ij}(X,g)\,.$$A {\it local formula} $P(g_{ij},g_{ij/\alpha})(X,g)$ is a polynomial in the $g_{ij/\alpha}$ variables withcoefficients which are smooth in the $g_{ij}$ variables where $\{g_{ij}\}$ is assumed to define a non-singular bilinear
form of signature $(p,q)$.  We say that $P$ is {\it invariant} if $P(X,g)(\xi)$ depends only on $(g,\xi)$ and not on the
particular coordinate system chosen; it is to be universally and polynomially defined in the category of
pseudo-Riemannian manifolds of signature
$(p,q)$. We let
$\mathcal{I}(p,q)$ be the vector space of all such invariants. The space of symmetric $2$-tensor invariants
$\mathcal{I}^{(2)}(p,q)$ is defined similarly.

There is a natural grading which is defined on $\mathcal{I}(p,q)$ and on $\mathcal{I}^{(2)}(p,q)$ by setting
$\operatorname{ord}(g_{ij/\alpha}):=|\alpha|$. We may then decompose
$$\mathcal{I}(p,q)=\oplus_\ell\mathcal{I}_{2\ell}(p,q)\quad\text{and}\quad
  \mathcal{I}^{(2)}(p,q)=\oplus_\ell\mathcal{I}_{2\ell}^{(2)}(p,q)$$
where the spaces $\mathcal{I}_{2\ell}(p,q)$ and $\mathcal{I}_{2\ell}^{(2)}(p,q)$ consist of invariants which are of order
$2\ell$ in the derivatives of the metric; there are no invariants of odd order. This grading can also be expressed in a
coordinate free fashion by noting that if $P\in\mathcal{I}(p,q)$ and if $P^{(2)}\in\mathcal{I}^{(2)}(p,q)$, then\begin{eqnarray*}
&&P\in\mathcal{I}_{2\ell}(p,q)\quad\Leftrightarrow\quad P(c^2g)=c^{-2\ell}P(g),\\&&P^{(2)}\in\mathcal{I}_{2\ell}^{(2)}(p,q)\quad\Leftrightarrow\quad P^{(2)}(c^2g)=c^{-2\ell-2}P^{(2)}(g)\,.
\end{eqnarray*}
Furthermore, it is clear
that
$$\mathcal{J}_{2\ell}(p,q)\subset\mathcal{I}_{2\ell}(p,q)\quad\text{and}\quad
   \mathcal{J}_{2\ell}^{(2)}(p,q)\subset\mathcal{I}_{2\ell}^{(2)}(p,q)\,.$$
On the other hand, the invariant
$$g^{il}g^{jk}R_{ijkl;ab}dx^a\circ dx^b$$
is an element of $\mathcal{I}_4^{(2)}(p,q)$ which does not belong to $\mathcal{J}_4^{(2)}(p,q)$ since it
involves the
$4^{\operatorname{th}}$ derivatives of the metric.

H. Weyl's theorem \cite{W46} shows that all the elements of
$\mathcal{I}_{2\ell}(p,q)$ and of $\mathcal{I}_{2\ell}^{(2)}(p,q)$
are given by suitable contractions of indices involving covariant
derivatives of the curvature tensor. We shall not introduce the
necessary formalism as it plays no role in our analysis. The
discussion of the restriction map in the geometric context given in
Section~\ref{sect-1.5} then extends the restriction maps defined
previously to define maps:
\begin{eqnarray*}
&&r_-^{(2)}:\mathcal{I}_{2\ell}^{(2)}(p,q)\rightarrow\mathcal{I}_{2\ell}^{(2)}(p-1,q)\quad\text{if}\quad p>0\,,\\
&&r_+^{(2)}:\mathcal{I}_{2\ell}^{(2)}(p,q)\rightarrow\mathcal{I}_{2\ell}^{(2)}(p,q-1)\quad\text{if}\quad q>0\,.
\end{eqnarray*}
These maps are characterized, as previously, by the property that
$$(r_\pm^{(2)} P_{2\ell}^{(2)})(N,g_N)(\xi)=
i_N^*\{P_{2\ell}^{(2)}(N\times S^1,g_N\pm d\theta^2)(\xi,\theta_1)\}\,.$$
Again, if we express the invariants in terms of contractions of indices, the restriction maps simply let the range of
summation be from $1$ to $m-1$ rather than from $1$ to $m$.

Fix a signature $(p,q)$ and let $0\ne
P\in\mathcal{I}_{2\ell}^{(2)}(p,q)$. Let $\varepsilon$ be a given
quadratic form of signature $(p,q)$. Let $\xi$ be a point of a
pseudo-Riemannian manifold $(M,g)$ of signature $(p,q)$. We can
choose coordinates $x=(x^1,...,x^m)$ which are centered at
$\xi$ (i.e. $x(\xi)=0$) and which are normalized so that
$$g_{ij}=\varepsilon_{ij}+O(|x|^2)\quad\text{where}\quad\varepsilon_{ij}=\left\{\begin{array}{lll}
\pm1&\text{if}&i=j\\0&\text{if}&i\ne j\end{array}\right\}\,.$$
With this normalization, the tensor $g_{ij}$ and the first derivatives of the metric play no role and we can regard
$P^{(2)}=P^{(2)}(g_{ij/\alpha})$ as a polynomial in the derivatives (where $|\alpha|\ge2$) which is symmetric
$2$-tensor valued. Thus a typical monomial takes the form
\begin{equation}\label{eqn-5.a}
A^{(2)}=g_{i_1j_1/\alpha_1}...g_{i_rj_r/\alpha_r}dx^{k}\circ
dx^{l}\,.
\end{equation}
We let $c(A^{(2)},P^{(2)})$ be the coefficient of $A^{(2)}$ in
$P^{(2)}$. We say $A^{(2)}$ {\it is a monomial of P} if
$c(A^{(2)},P^{(2)}{{)}}\ne0$. We let
$\operatorname{deg}_n(A^{(2)})$ be the number of times that the
index $n$ appears in the monomial $A^{(2)}$:
\begin{equation}\label{eqn-5.b}
\operatorname{deg}_n(A^{(2)})=\delta_{i_1,n}+\delta_{j_1,n}+\alpha_1(n)+...+\delta_{i_r,n}+\delta_{j_r,n}+\alpha_r(n)
+\delta_{k,n}+\delta_{l,n}
\end{equation}

\begin{lemma}\label{lem-5.1}
Let $0\ne P_{2\ell}^{(2)}\in\mathcal{I}_{2\ell}^{(2)}(p,q)$. Assume that $r_-^{(2)}(P_{2\ell}^{(2)})=0$ if $p>0$ and that
$r_+^{(2)}(P_{2\ell}^{(2)})=0$ if
$q>0$.
\begin{enumerate}
\item $\operatorname{deg}_n(A^{(2)})\ge2$ for every monomial $A^{(2)}$ of $P_{2\ell}^{(2)}$.
\smallbreak\item $p+q\le 2\ell+1$.
\smallbreak\item If $p+q=2\ell+1$, then $P_{2\ell}^{(2)}\in\mathcal{J}_{2\ell}^{(2)}(p,q)$.
\end{enumerate}
\end{lemma}

\begin{proof} If $p>0$, we choose an orthonormal basis for the model space $\{e_i\}$ so $e_m$ is timelike. Then
$r_-^{(2)}(P_{2\ell}^{(2)})$ is defined by evaluating
$P_{2\ell}^{(2)}$ on a metric of the form $ds^{(2)}_N-d\theta^2$.
The only additional relation that is imposed by restricting our
attention to such metrics is to set $A^{(2)}=0$ if
$\operatorname{deg}_m(A^{(2)})>0$. Thus we may conclude that
$\operatorname{deg}_m(A^{(2)})>0$ for every monomial $A^{(2)}$ of
$P_{2\ell}^{(2)}$. By replacing $\partial_{x_m}$ by
$-\partial_{x_m}$, we may conclude that
$\operatorname{deg}_m(A^{(2)})$ is even and thus
$\operatorname{deg}_m(A^{(2)})\ge 2$ for every monomial $A^{(2)}$ of
$P_{2\ell}^{(2)}$. We can permute the coordinate indices;
this may replace a space-like coordinate vector field by a time-like
coordinate vector field and thus it is important to work with
$r_+^{(2)}$ and $r_-^{(2)}$ simultaneously if both $p$ and $q$ are
positive and it was for this reason that we introduced both
$r_-^{(2)}$ and $r_+^{(2)}$. Thus we have
$$\operatorname{deg}_k(A^{(2)})\ge2\quad\text{for}\quad 1\le k\le p+q\,.$$
We adopt the notation of Equation~(\ref{eqn-5.a}) and use
Equation~(\ref{eqn-5.b}) together with the fact that
$|\alpha_j|\ge2$ to estimate:
\begin{eqnarray*}
&&2(p+q)\le\sum_{n=1}^{p+q}\operatorname{deg}_n(A^{(2)})=2r+2+\sum_{j=1}^r|\alpha_j|
\le 2+2\sum_{j=1}^r|\alpha_j|=2+4\ell\,.
\end{eqnarray*}

Assertion (2) now follows. In the limiting case that $p+q=2\ell+1$,
all of the equalities must have been equalities. In particular,
$|\alpha_j|=2$ for all $j$. This establishes Assertion (3).
\end{proof}

\subsection{The proof of Theorem~\ref{thm-1.9}} It is clear from the definition of the Euler-Lagrange equations that since
$E_{2\ell}$ is universal, we have that $S_{2\ell,2\ell}^{(2)}$ is
universal as well. By Theorem~\ref{thm-1.5}
$S_{2\ell,2\ell}^{(2)}=0$. Thus
${{r^{(2)}}}(S_{2\ell,2\ell+1}^{(2)})=0$ so by Lemma~\ref{lem-5.1},
$$
   S_{2\ell,2\ell+1}^{(2)}\in\mathcal{J}_{2\ell,2\ell+1}\,.
$$
We may therefore apply Theorem~\ref{thm-1.8} to conclude
$$
   S_{2\ell,2\ell+1}^{(2)}=cT_{2\ell,2\ell+1}^{(2)}
$$
for some constant $c$. Therefore $
r^{(2)}(S_{2\ell,2\ell+2}^{(2)}-cT_{2\ell,2\ell+2}^{(2)})=0$ so
 by Lemma~\ref{lem-5.1},
$$
    S_{2\ell,2\ell+2}^{(2)}-cT_{2\ell,2\ell+2}^{(2)}=0\,.
$$
We consider in this fashion to show
$$
   S_{2\ell,m}^{(2)}-cT_{2\ell,m}^{(2)}=0\quad\text{for}\quad m\ge 2\ell+1\,.
$$
The equality if $m<2\ell+1$ is of course trivial since both $S_{2\ell,m}^{(2)}$ and $T_{2\ell,m}^{(2)}$ vanish
in that instance.\hfill\qed

\section*{Acknowledgments}
Research of P. Gilkey partially supported by project MTM2009-07756
(Spain).  Research of {J.H.} Park and K. Sekigawa was supported by
the National Research Foundation of Korea (NRF) grant funded by the
Korea government (MEST) (2011-0012987).

\end{document}